\documentclass[]{amsart}
\usepackage{amsmath,amscd,amsfonts,mathtools,amsthm,amssymb}
\usepackage[colorlinks, linkcolor=blue, citecolor=red]{hyperref}
\theoremstyle{plain}
\newtheorem{theorem}{Theorem}[section]
\newtheorem{lemma}[theorem]{Lemma}
\newtheorem{corollary}[theorem]{Corollary}
\newtheorem{proposition}[theorem]{Proposition}

\theoremstyle{definition}
\newtheorem{definition}[theorem]{Definition}

\theoremstyle{remark}
\newtheorem{remark}[theorem]{Remark}

\newtheorem{question}[theorem]{Question}
\newtheorem{conjecture}[theorem]{Conjecture}

\def\Nrd{{\rm Nrd}} 

\def\Int{{\rm Int}}

\def\R{{\mathbb{R}}}

\def\Q{{\mathbb{Q}}}

\title[Weak isotropy over totally positive field extensions]{Weak isotropy of central simple algebras with orthogonal involutions over totally positive field extensions}
\author[P. Mandal]{Priyabrata Mandal}
\address{Department of  Mathematics, Bioinformatics and Computer Applications,  Maulana Azad National Institute of Technology Bhopal, Bhopal-462003, India}
\email{priyabrata@manit.ac.in}
\author[A. Soman]{Abhay Soman}
\address{School of Mathematics and Statistics, University of Hyderabad, Hyderabad-500046, India}
\email{abhaysoman@uohyd.ac.in}

\date{}

\begin{document}

\begin{abstract}
In this paper, we explore the behavior of orthogonal involutions in the context of totally positive field extensions. Let $K/F$ be a totally positive extension of formally real fields. By Becher's result, if a quadratic form $q$ over $F$ becomes isotropic over $K$, then $q$ is weakly isotropic over $F$. We present an example in which, despite $K/F$ being totally positive, a central simple algebra $(A, \sigma)$ over $F$ with an orthogonal involution becomes isotropic over $K$, while remaining strongly anisotropic over $F$. However, when $K/F$ is assumed to be a Galois totally positive $2$-extension of formally real fields, we show that an analogue of Becher’s result for quadratic forms holds for orthogonal involutions. Furthermore, for a totally positive Galois field extension $K/F$, we verify Becher’s conjecture for central division algebras of index $2^n$ and exponent $2$ containing a subfield of $F_{\rm py}$ of degree $2^{n-2}$ over $F$.
\end{abstract}
\maketitle
\section {Introduction}

    In this article, we assume that all fields are formally real. For basic definitions and notations, we refer the reader to \S\ref{prelim}. Let $F$ be a formally real field and let $F_{\rm py}$ be the Pythagorean closure of $F$ in a fixed algebraic closure of $F$. In \cite{Becher}, Becher defines the notion of a totally positive field extension, viz., an extension of formally real fields $K/F$ is said to be \emph{totally positive} if every semiordering on $F$ can be extended to a semiordering on $K$. Furthermore, he proved that $K/F$ is a totally positive field extension if and only if a quadratic form $q$ defined over $F$ becomes isotropic over $K$, then $q$ is weakly isotropic over $F$. This leads us to ask the following question.
    \begin{question}\label{question-going-down-weak-isotropy}
        Let $K/F$ be a totally positive field extension of formally real fields and let $(A,\sigma)$ be a central simple algebra over $F$ with an orthogonal involution. If $\sigma$ becomes isotropic over $K$, does it follow that $\sigma$ is weakly isotropic over $F$? 
    \end{question}
    We answer this question negatively in Theorem \ref{weakly isotropic involutions under totally positive extensions does not go down}. However, if $K/F$ is assumed to be a finite Galois $2$-extension, then the above Question \ref{question-going-down-weak-isotropy} has an affirmative answer, see Theorem \ref{going down weak isotropy for tp galois 2-extnesions}. More precisely, we obtain the following result.
    \begin{theorem}[= Theorem \ref{going down weak isotropy for tp galois 2-extnesions}]\label{intro-going down weak isotropy for tp galois 2-extnesions}
        Let $K/F$ be a finite Galois $2$-extension of formally real fields that is totally positive. Let $(A, \sigma)$ be a central simple algebra with an orthogonal involution $\sigma$ over $F$. If $\sigma$ is weakly isotopic over $K$, then it is also weakly isotropic over $F$.
    \end{theorem}
    In \cite{Becher}, Becher conjectured that if $K/F$ is a totally positive field extension of formally real fields, then the Pythagorean index of a central simple algebra $A$ over $F$ of exponent $2$ is the same as the Pythagorean index of $A\otimes_F K$. We verify the conjecture in a special case of central division algebras containing subfields of $F_{\rm py}$ of certain degrees. More precisely, we have the following.
    \begin{proposition}[= Proposition \ref{becher conjecture particular case}]\label{intro-becher conjecture particular case}
	   Let $F$ be a formally real field, and let $K/F$ a Galois totally positive field extension. Suppose $D$ is a central division algebra of the index $2^n$ and exponent $2$ that contains a subfield $L\subset F_{\rm py}$ with $[L:F]=2^{n-2}$. Then, $pind(D)=pind(D\otimes_FK)$, i.e., Becher's conjecture holds for such central division algebras.
    \end{proposition}

    Recall that a weakly isotropic quadratic form over a Pythagorean field is necessarily isotropic. However, we show that the analogue of this result does not hold for central simple algebras with orthogonal involutions over Pythagorean fields (see Proposition \ref{weak isotropy does not imply isotropy for involution}). In a similar spirit, we construct an example of a Pythagorean field $E$ and a quaternion central division algebra with an orthogonal involution over $E$ where the sum of hermitian squares fails to be a hermitian square (see Proposition \ref{sum of hermitian squares need not be a hermitian square over Pythagorean field}).

\section{Preliminaries}\label{prelim}

In this section, we review some fundamental definitions and concepts. For a more detailed explanation, we refer the reader to the relevant sections in \cite{Lam, KMRT}.
 
A field $F$ is said to be \textit{formally real} if $-1$ cannot be expressed as a sum of squares in $F$. A field $F$ is called \textit{Pythagorean} if every sum of squares in $F$ is itself a square. The {\it Pythagorean closure} of $F$,  denoted by $F_{\rm py}$, is the smallest subfield of $F_{\rm al}$ containing $F$ that is Pythagorean. Next, we recall the definition of an ordering on a field. 
We denote by $F^*$ the set of nonzero elements of $F$.
Throughout, by a \emph{quadratic form} we mean a \emph{nondegenerate} quadratic form.
\begin{definition}
An ordering $P$ on a field $F$ is a proper subset $P \subsetneq F$ satisfying the following conditions:  
$F^2 \subseteq P$,  $P + P \subseteq P$,  $P \cdot P \subseteq P$,  $P \cup -P = F$,  and  $P \cap -P = \{0\}$.  

\end{definition}
\medskip 

According to (\cite[Theorem $1.10$, Chapter $8$]{Lam}), a field $F$ is formally real if and only if it has at least one ordering.  
Let $X_F$ denote the set of all orderings of $F$. Let $F_{\rm al}$ be an algebraic closure of $F$.  
For any ordering $P \in X_F$, there exists a real closure $F_P$ of $(F, P)$ in $F_{\rm al}$, which is unique up to isomorphism (see \cite[Theorem $2.8$, Chapter $8$]{Lam}).  

\medskip 

\begin{definition} (see \cite[page $5$]{Prestel}) 
\label{semiorder}
	A semiordering on a field $F$ is a subset $S\subset F$ satisfying the following conditions: $1\in S$, $F^2S\subset S$, $S+S\subset S$, $S\cup -S=F$, and $S\cap -S=\{0\}$.
\end{definition}
\medskip 

\begin{definition}[Totally positive field extension]\label{def-totally-positive}
    A field extension $K/F$ of formally real fields is said to be \emph{totally positive} if every semiordering on $F$ extends to a semiordering on $K$.
\end{definition}
Becher provides the following useful characterization of total positiveness using quadratic forms.
\begin{lemma}\cite[Lemma 3.1]{Becher}\label{equivalent characterization of totally positive field extensions}
    Let $K/F$ be a field extension of formally real fields. The following statements are equivalent.
    \begin{enumerate}
        \item $K/F$ is totally positive.
        \item If a quadratic form $q$ defined over $F$ is isotropic over $K$, then $q$ is weakly isotropic over $F$.
    \end{enumerate}
\end{lemma}
Next, we recall the Effective Diagonalization Property from \cite{Ware-Nagoya}.
\begin{definition}
    A quadratic form $\langle a_1, \dots, a_n \rangle$ is \textit{effectively diagonalizable} if it is isometric to a form $\langle b_1, \dots, b_n \rangle$ satisfying $b_i \in P \implies b_{i+1} \in P$ for all $1 \leq i < n$ and all $P \in X_F$. The field $F$ satisfies the \textit{Effective Diagonalization Property} (or is \textit{ED}, for short) if every quadratic form over $F$ is effectively diagonalizable.
\end{definition}

Next we recall the definition of an involution.
\begin{definition} \cite[\S 2]{KMRT} \label{involution definition}
An \textit{involution} on a central simple algebra $A$ over a field $F$ is a map $\sigma: A \to A$ that satisfies the following conditions:
\begin{enumerate}
    \item $\sigma(x + y) = \sigma(x) + \sigma(y)$ for all $x, y \in A$,
    \item $\sigma(\lambda x) = \lambda \sigma(x)$ for all $\lambda \in F$ and $x \in A$,
    \item $\sigma(xy) = \sigma(y)\sigma(x)$ for all $x, y \in A$,
    \item $\sigma(\sigma(x)) = x$ for all $x \in A$.
\end{enumerate}
    
\end{definition}
\medskip 

From Definition \ref{involution definition}, we see that the center $F$ is preserved under the map $\sigma$, i.e., $\sigma(F) = F$.  
Therefore, the restriction of $\sigma$ to $F$ is either the identity map or an automorphism of order $2$.  
Involutions that leave the center elementwise invariant are called \textit{involutions of the first kind}.  
If $K/F$ is a field extension and $A$ is a central simple algebra over $F$ with an involution $\sigma$, then we denote by $\sigma_K$ the involution $\sigma \otimes 1_K$ on the central simple algebra $A \otimes_F K$ over $K$.  

Let $A$ be a central simple algebra over $F$. Consider the case when $A$ splits, i.e., $A \cong \text{End}_F(V)$ for some vector space $V$ over $F$.  
Let $b$ be a non-singular bilinear form on $V$. Then, for any $f \in \text{End}_F(V)$, one can define a map  
$\sigma_b : \text{End}_F(V) \to \text{End}_F(V)$ satisfying the following condition:  

\[
b(x, f(y)) = b(\sigma_b(f)(x), y), \quad \text{for all } x, y \in V.
\]

The map $\sigma_b$ is called the \textit{adjoint involution} with respect to the non-singular bilinear form $b$.  
The $F$-linear involutions on $\text{End}_F(V)$ are precisely the adjoint involutions with respect to symmetric or skew-symmetric non-singular bilinear forms on $V$ (see \cite[Chapter 1]{KMRT}).

\begin{definition} \cite[\S 2]{KMRT}
An involution $\sigma$ of the first kind is said to be of \textit{orthogonal type} (or simply \textit{orthogonal}) if for any splitting field $L$ and any isomorphism $(A_L, \sigma_L) \simeq (\text{End}_L(V), \sigma_b)$, the bilinear form $b$ is symmetric.
    
\end{definition}

\begin{definition}
  Let $(A, \sigma)$ be a central simple algebra with an orthogonal involution $\sigma$ over a field $F$.  
The involution $\sigma$ is said to be \emph{isotropic} if there exists a nonzero $x \in A$ such that $\sigma(x)x = 0$.  
The involution is said to be \emph{weakly isotropic} if there exist nonzero elements $x_1, x_2, \dots, x_n \in A$ such that  
\[
\sum_{i=1}^{n} \sigma(x_i) x_i = 0.
\]
If $\sigma$ is not weakly isotropic, then we say that $\sigma$ is \emph{strongly anisotropic}.  
In other words, $\sigma$ is \emph{strongly anisotropic}, if for any $n \in \mathbb{N}$, and any nonzero elements $x_1, x_2, \dots, x_n \in A$,  
\[
\sum_{i=1}^{n} \sigma(x_i) x_i \neq 0.
\]
\end{definition}

A \textit{generalized quaternion algebra} $Q = (a, b)_F$ is an $F$-algebra generated by elements $i$ and $j$ satisfying the relations  
\[
i^2 = a, \quad j^2 = b, \quad \text{and} \quad ij = -ji,
\]
where $\{1, i, j, k = ij\}$ forms a basis for $Q$ over $F$. This is a central simple algebra over $F$ of dimension $4$, and it has a canonical involution $\gamma$ on $Q$ defined by  
\[
\gamma(x_0 + x_1 i + x_2 j + x_3 k) = x_0 - x_1 i - x_2 j - x_3 k,  
\]
for $x_0, x_1, x_2, x_3 \in F$. When $F = \mathbb{R}$ and $a = b = -1$, we obtain the Hamiltonian algebra $\mathbb{H}$ over $\mathbb{R}$.  

\section{Main results}
\subsection{Algebras with involutions}
    
We begin with the following proposition, which follows similar lines as the proof of \cite[Proposition 3.9]{MPS}.  
We include it here for completeness.  

	\begin{proposition}\label{finite Galois tp}
	Let $K/F$ be a totally positive finite Galois extension of formally real fields and $E$ be a field such that $F \subset E \subset K$. Then $K/E$ is also a totally positive finite Galois extension. 
\end{proposition}
\begin{proof}
	Let $K = F(\theta)$, and let $f(t) \in F[t]$ be the minimal polynomial of $\theta$ over $F$.  
By \cite[3.5]{BLS} and \cite[VIII.2.15]{Lam}, the polynomial $f(t)$ has exactly $[K:F]$ real roots,  
i.e., all the roots of $f(t)$ lie in a real closure $F_P$ of $F$.  

Let $g(t)$ be the minimal polynomial of $\theta$ over $E$. Since  
$g(t)$ divides $f(t)$, all the roots of $g(t)$ also lie in $F_P$.  
Since $E/F$ is totally positive, there exists an ordering $\tilde{P} \in X_E$  
such that $\tilde{P}$ extends $P$. As $E/F$ is algebraic, we have  
$E_{\tilde{P}} = F_P$. Hence, all the roots of $g(t)$ lie in $E_{\tilde{P}}$.  
Furthermore, since $K/E$ is Galois, it is totally positive by  
\cite[VIII.2.15]{Lam} and \cite[3.9]{BLS}.  

\end{proof}

We recall the following result from \cite{LSU}.

\begin{lemma}[\cite{LSU}]\label{weak isotropy holds for quadratic extension} Let $A$ be a central simple $F$-algebra with an orthogonal involution $\sigma$ of the first kind. Let $K=F(\sqrt{d})$, where $d$ is a sum of squares in $F$. If $\sigma_K$ is weakly isotropic, then $\sigma$ is weakly isotropic.
\end{lemma}

We now show that for a \emph{Galois totally positive finite $2$-extension} of fields, the analogue of Lemma \ref{equivalent characterization of totally positive field extensions} holds for orthogonal involutions.
\begin{theorem}\label{going down weak isotropy for tp galois 2-extnesions}
   Let $K/F$ be a Galois totally positive finite $2$-extension of formally real fields.  
Let $(A, \sigma)$ be a central simple algebra with an orthogonal involution over $F$.  
If $\sigma_K$ is weakly isotropic, then $\sigma$ is weakly isotropic over $F$.  

\end{theorem}
\begin{proof}
    Suppose that $(A, \sigma)$ is strongly anisotropic over $F$. We consider the following set:  
\[
\Sigma = \left\{F_i/F \mid F \subseteq F_i \subseteq K \text{ and } \sigma_{F_i} \text{ is strongly anisotropic} \right\}.
\]

Since $F \in \Sigma$, the set $\Sigma$ is nonempty.  
The inclusion of fields defines a partial ordering on $\Sigma$.  
Assume that $\{F_i\}_{i\in I}$ is a chain in $\Sigma$.  
We claim that $\bigcup_{i\in I} F_i \in \Sigma$.  
Since $\{F_i\}_{i\in I}$ is a chain, the union $\bigcup_{i\in I} F_i$ is a field containing $F$ and contained in $K$.  
If $\sigma$ is weakly isotropic over $\bigcup_{i\in I} F_i$, then there exists some $k \in I$ such that $\sigma$ is weakly isotropic over $F_k$,  
which is a contradiction since $F_k \in \Sigma$.  
Therefore, we conclude that $\bigcup_{i\in I} F_i \in \Sigma$.  
Hence by Zorn's lemma, there exists a maximal element, say $E\in\Sigma$.

   Note that $E \neq F$. Indeed, suppose that $E = F$. Since ${\rm Gal}(K/F)$ is a $2$-group, there exists a normal subgroup $H \trianglelefteq {\rm Gal}(K/F)$ such that $K^H$ is a quadratic field extension of $F$.  
Furthermore, since $K/F$ is totally positive, the extension $K^H/F$ is also totally positive.  
By Lemma \ref{weak isotropy holds for quadratic extension}, $\sigma$ is strongly anisotropic over $K^H$,  
contradicting the maximality of $E$.  

So, we assume that $E \neq F$ and that $[K:F] = 2^n$. We proceed by induction on $n$.  
If $n = 1$, then $E = K$, and we are done. Suppose that $E \neq K$ and that $[K:E] = 2^m$ with $m < n$.  
Now, $K/E$ is a Galois $2$-extension, and by Proposition \ref{finite Galois tp}, it is totally positive.  
By the induction hypothesis, $\sigma_K$ is strongly anisotropic.  

\end{proof}

However, we now show that the analogue of Lemma \ref{equivalent characterization of totally positive field extensions} \emph{does not} hold for extensions of orthogonal involutions over \emph{arbitrary} totally positive field extensions in Theorem \ref{weakly isotropic involutions under totally positive extensions does not go down}.  

To proceed, we require \emph{a part of a theorem} due to Prestel and Ware. For completeness, we state the theorem below.
\begin{theorem}\cite[Theorem 2]{Prestel-Ware}\label{Prestle-Ware-result}
    For a formally real field $F$ the following statements are equivalent.
    \begin{enumerate}
        \item $F$ is SAP and every torsion quadratic form $\langle a,b\rangle$ over $F$ represents a totally positive element of $F^*$.
        \item For all $a,b\in F^*$, the form $\langle 1,a,b,-ab\rangle$ is almost isotropic.
    \end{enumerate}
\end{theorem}
We refer to \cite[page 241]{Prestel-Ware} for the definition of an \emph{almost isotropic quadratic form}. We note that every \emph{isotropic} quadratic form is \emph{almost isotropic}.

 \begin{lemma}\label{laurent series over R is ED}
		The Laurent series field over $\R$ in one variable, $\R((t))$ is ED.
	\end{lemma}
	\begin{proof}
		Put $F = \mathbb{R}((t))$. We show that for any $a, b \in F^*$, the quadratic form $\langle 1, a, b, -ab \rangle$ is \emph{isotropic}.  

        The square classes of $F$ satisfy $F^*/F^{*2} = \{\pm 1, \pm t\}$, and $\{-1, t\}$ forms a $\mathbb{Z}/2\mathbb{Z}$-basis of $F^*/F^{*2}$ (see \cite[VI.1.2]{Lam}). Therefore, the form $\langle 1, a, b, -ab \rangle$ is isotropic, and thus, by Theorem \ref{Prestle-Ware-result}, we conclude that $F$ is ED.
	\end{proof}
	
	\begin{proposition}
		The field extension $\R((t))\big/\Q((t))$ is totally positive.
	\end{proposition}
	\begin{proof}
		Let $q$ be a quadratic form over $\mathbb{Q}((t))$ that is isotropic over $\mathbb{R}((t))$. By Lemma \ref{equivalent characterization of totally positive field extensions}, it is enough to show that $q$ is weakly isotropic over $\mathbb{Q}((t))$. 
        By \cite[VI.1.5]{Lam}, there exist quadratic forms $q_1 = \langle u_1, u_2, \ldots, u_r \rangle$ and $q_2 = \langle u_{r+1}, \ldots, u_n \rangle$ such that the elements $u_i$ are units in the $t$-adic discrete valuation ring of $\mathbb{Q}((t))$, and $q \simeq q_1 \perp \langle t \rangle q_2$.  
        The $t$-adic discrete valuation on $\mathbb{R}((t))$ is an extension of the $t$-adic discrete valuation on $\mathbb{Q}((t))$. Using \cite[VI.1.9]{Lam}, the quadratic form $q$ is isotropic over $\mathbb{R}((t))$ if and only if either $\langle \overline{u_1}, \ldots, \overline{u_r} \rangle$ or $\langle \overline{u_{r+1}}, \ldots, \overline{u_n} \rangle$ is isotropic over the residue field $\overline{\mathbb{R}((t))} = \mathbb{R}$.

        As $\mathbb{Q}$ has a unique ordering that can be extended to $\mathbb{R}$, the field extension $\mathbb{R}/\mathbb{Q}$ is totally positive (see \cite[Proposition 4.2]{MPS}). Hence, either $\langle \overline{u_1}, \ldots, \overline{u_r} \rangle$ or $\langle \overline{u_{r+1}}, \ldots, \overline{u_n} \rangle$ is weakly isotropic over $\mathbb{Q} = \overline{\mathbb{Q}((t))}$. Thus, by \cite[VI.1.9]{Lam}, $q$ is weakly isotropic over $\mathbb{Q}((t))$.

	\end{proof}

    In the following theorem, we show that the analogue of Lemma \ref{equivalent characterization of totally positive field extensions} does not hold for orthogonal involutions. Hence, in general, the answer to Question \ref{question-going-down-weak-isotropy} is negative.
    
	\begin{theorem}\label{weakly isotropic involutions under totally positive extensions does not go down}
		There exists a central simple algebra $A$ over $\mathbb{Q}((t))$ with a strongly anisotropic orthogonal involution $\sigma$ that becomes weakly isotropic over $\mathbb{R}((t))$.

	\end{theorem}
	\begin{proof}
		By \cite[Example 3.4]{LSU}, $\mathbb{Q}((t))$ is SAP but not ED. Following \cite[Theorem 3.12]{LSU}, we take $Q = (2,t)$ over $\mathbb{Q}((t))$, a quaternion algebra, and a skew-hermitian form $h = \langle j, k \rangle$ with respect to the canonical involution $\gamma$ on $Q$.  

        Let $L = \mathbb{Q}((t))(\sqrt{2})$, and define $h_1 = \langle j \rangle$ and $h_2 = \langle k \rangle$. We have the following:  
        \[
            h_1(1,1) = j; \quad h_1(1,j) = t; \quad h_1(j,j) = \gamma(j)jj = -tj.
        \]

		Similarly, we have  
        \[
            h_2(1,1) = k = ij; \quad h_2(1,j) = ti; \quad h_2(j,j) = tij.
        \]  
        By \cite[\S10.3]{S}, there exists a map $\pi_2 \colon Q \to L$ given by $\alpha + \beta j \mapsto \beta$, where $\alpha, \beta \in L$. Therefore, the matrix representations of $\pi_2(h_i)$ are  
        \[
            \begin{pmatrix}
            	1 & 0 \\ 
            	0 & -t
            \end{pmatrix}
            \quad \text{and} \quad
            \begin{pmatrix}
            	i & 0 \\ 
            	0 & ti
            \end{pmatrix},
        \]
        respectively.
        Thus, over $L$, the hermitian form $h = \langle j, k \rangle = h_1 \perp h_2$ corresponds to a quadratic form  
        \[
            \langle 1, -t, \sqrt{2}, t\sqrt{2} \rangle.
        \]  
        This quadratic form is totally indefinite. Since $L/\mathbb{Q}((t))$ is a totally positive quadratic field extension, it follows that the involution corresponding to the Hermitian form $h$ on $M_2(Q)$ is totally indefinite as well (see \cite[Corollary 11.11(2(b))]{KMRT}). By Lemma \ref{laurent series over R is ED} and \cite[Theorem 3.8]{LSU}, $\sigma$ is weakly isotropic over $\mathbb{R}((t))$.
	\end{proof}

    \begin{remark}
        If $E/F$ is a totally positive field extension and $F$ is ED, then the answer to Question \ref{question-going-down-weak-isotropy} is affirmative (see \cite[Corollary 5.4]{MPS}). Furthermore, for \emph{totally decomposable} central simple algebras with orthogonal involutions, the answer to Question \ref{question-going-down-weak-isotropy} is also affirmative, see \cite[Theorem 1.3]{MPS}.
    \end{remark}

    In the rest of this section we collect a few results that are of the independent interest.

    Let $(A, \sigma)$ be a central simple algebra over a formally real field $F$ with an orthogonal involution, and suppose that $\sigma$ is strongly anisotropic. The following proposition shows the existence of a formally real Pythagorean field $K$ such that $\sigma_K$ remains strongly anisotropic, and the field $K$ so obtained is maximal among algebraic extensions of $F$ over which $\sigma$ remains strongly anisotropic. More precisely, we have the following.

    \begin{proposition}\label{Zorn's lemma for extension}
        Let $(A, \sigma)$ be a central simple algebra over a formally real field $F$ with an orthogonal involution. Assume that $\sigma$ is strongly anisotropic. Then, there exists a formally real Pythagorean algebraic field extension $K/F$ such that $\sigma_K$ remains strongly anisotropic, and for any proper algebraic field extension $L/K$, the involution $\sigma_L$ is weakly isotropic.
    \end{proposition}
    \begin{proof}
	   We fix an algebraic closure $F_{\rm al}$ of $F$. Consider the following set:  
            \[
                S = \{E/F: F \subseteq E \subseteq F_{\rm al} \text{ and } \sigma_E \text{ is strongly anisotropic} \}.
            \]

        By our assumption, $F \in S$. The inclusion of fields defines a partial ordering on $S$. Consider a chain $\{E_i\}$ in $S$. Then $E = \cup_i E_i$ is a field extension of $F$. It is easy to see that $\sigma_E$ is strongly anisotropic, and hence $E \in S$. By Zorn's lemma, $S$ has a maximal element, say $K$. 
        Note that $K$ is formally real. Indeed, if there exist elements $\beta_i \in K^*$ such that $\sum \beta_i^2 = 0$, then  
            \[
                \sum \sigma_K(\beta_i) \beta_i = \sum \beta_i^2 = 0,
            \]  
        i.e., $\sigma_K$ is weakly isotropic, a contradiction. Hence, $K$ is formally real.

        Now we show that $K$ is a Pythagorean field. Let $a \in K \setminus \{0\}$ be such that $a \in \sum K^2$, i.e., $a$ is totally positive. By Lemma \ref{weak isotropy holds for quadratic extension}, $\sigma_{K(\sqrt{a})}$ is strongly anisotropic. By the maximality of $K$, we must have $K = K(\sqrt{a})$, i.e., $a \in K^2$. In particular, $\sum K^2 \subset K^2$, which implies that $K$ is Pythagorean.
    \end{proof}

The following result follows from the above Proposition \ref{Zorn's lemma for extension}. 
    \begin{corollary}\label{weak isotropy holds for Pythagorean closure}
        Let $A$ be a central simple algebra over $F$ with an orthogonal involution $\sigma$. If $\sigma$ is weakly isotropic over $F_{\rm{py}}$, then $\sigma$ is weakly isotropic.
    \end{corollary}

Recall that over a Pythagorean field every weakly isotropic quadratic form is isotropic. We show that the analogue of this result does not hold for central simple algebras with orthogonal involutions over Pythagorean fields.

    \begin{proposition}\label{weak isotropy does not imply isotropy for involution}
	   Let $\mathbb{H} = (-1,-1)_{\mathbb{R}}$ be the Hamiltonian division algebra over $\mathbb{R}$ with the orthogonal involution $\sigma = \Int(i) \circ \gamma$. Then, the involution $\sigma$ is anisotropic but weakly isotropic.
    \end{proposition}

    \begin{proof}
        A straightforward computation shows that $\sigma$ is anisotropic.We now show that $\sigma$ is weakly isotropic.  

        By \cite[Corollary 11.11 (1a)]{KMRT}, we have $sgn(\sigma) = 0$. Hence, $(\mathbb{H}, \sigma)$ is totally indefinite. We note that $\mathbb{R}$ is an ED field.  
        Indeed, given any regular quadratic form $q$ over $\mathbb{R}$, there exist $m, n \in \mathbb{N}$ such that $q \simeq m\langle -1 \rangle \perp n\langle 1 \rangle$. Hence, $q$ is effectively diagonalizable. By \cite[Theorem 3.8]{LSU}, $\mathbb{R}$ satisfies (WHA), i.e., every totally indefinite algebra with an involution of the first kind over $\mathbb{R}$ is weakly isotropic. In particular, $(\mathbb{H}, \sigma)$ must be weakly isotropic.	
    \end{proof}

The following result shows that the sum of hermitian squares of an orthogonal involution need not be a hermitian square over a Pythagorean field.

    \begin{proposition}\label{sum of hermitian squares need not be a hermitian square over Pythagorean field}
        Let $Q = (t, t)$ be a quaternion central division algebra over $F((t))$, where   $F$ is a formally real Pythagorean field. Let $\sigma = {\rm Int}(i) \circ \gamma$ be an orthogonal involution on $Q$, where $\gamma$ is the canonical involution on $Q$. Then, there exists a sum of hermitian squares that is not itself a hermitian square in $(Q, \sigma)$.  
	\end{proposition}
	\begin{proof}
		As $F$ is formally real and Pythagorean, so is $F((t))$ (see \cite[VIII.4.11]{Lam}). Note that $\sigma(i) = -i$, $\sigma(j) = j$, and $\sigma(k) = \sigma(j) \sigma(i) = -ji = ij = k$. Moreover, we have $k^2 = i(ji)j = -t^2$.
         We obtain the following:  
            \[
                i \cdot \sigma(i) + (1 + j) \cdot \sigma(1 + j) + t^{-1} k \cdot \sigma(t^{-1} k) = 2j.
            \]  
        If there exists $x \in Q$ such that $2j = x\sigma(x)$, then the reduced norm $\Nrd(2j) = (\Nrd(x))^2$, i.e., $-4t = (\Nrd(x))^2 \in F((t))$. This implies that $-t$ is a square in $F((t))$, and hence $-t$ must be positive with respect to every ordering on $F((t))$ (see \cite[VIII.1.12]{Lam}).  
        However, there exists an ordering $P$ on $F((t))$, extending an ordering on $F$, such that $-t$ is negative with respect to $P$. Hence, $2j$ cannot be a hermitian square.  
	\end{proof}

 \subsection{Subfields of division algebras}

In this section, we discuss a consequence of Becher's conjecture when the base field is Pythagorean (see Proposition \ref{no-totally-positive-subfield}). Furthermore, in Proposition \ref{becher conjecture particular case}, we prove a particular case of the conjecture. We begin by recalling the conjecture.

    \begin{conjecture}\cite[Conjecture 4.1]{Becher}\label{Becher's conjecture}
	Let $K/F$ be a totally positive field extension, and $D$ a central division algebra over $F$ of the exponent $2$. Then $pind(D)=pind(D\otimes K)$.
    \end{conjecture}
We recall the following result due to Marshall from \cite[Theorem 4, Corollary]{Marshall}.

    \begin{theorem}[Marshall's result]\label{Marshall's result}
		The kernel of the map \[{}_2Br(F)\to\prod_{P\in X_F}{}_2Br(F_P)\] is generated by the classes of quaternion algebras of the form $(s,t)_F$ with $s\in \sum F^2$ and $t\in F^*$.
    \end{theorem}

\begin{lemma}
Let $F$ be a formally real field, and let $D$ be a central division algebra over $F$ of exponent $2$. If there exists a totally positive maximal subfield $K$ of $D$, then $\operatorname{pind}(D) = 1$.  

\end{lemma}
\begin{proof}
	We have $D \otimes_F K \sim K \in \operatorname{Br}(K)$ since $K$ is a maximal subfield of $D$. Since $K/F$ is a finite field extension and $K/F$ is totally positive, we have $K_P = F_P$ for every $P \in X_F$. Thus, for every ordering $P \in X_F$, it follows that $D \otimes_F F_P \sim F_P \in \operatorname{Br}(F_P)$. By Marshall's result (Theorem \ref{Marshall's result}), we obtain $D \otimes_F F_{\rm py} \sim F_{\rm py}$, and hence $\operatorname{pind}(D) = 1$.  

\end{proof}
\begin{corollary}
Let $D$ be a central division algebra of exponent $2$. If $\operatorname{pind}(D) \neq 1$, then $D$ does not have a totally positive maximal subfield.  

\end{corollary}
%
Now, we discuss a consequence of Becher’s conjecture when the base field is Pythagorean.  

\begin{proposition}\label{no-totally-positive-subfield}  
If Becher's conjecture holds, then a central simple algebra over a formally real \emph{Pythagorean} field $F$ of exponent $2$ cannot have a totally positive subfield other than $F$.  
\end{proposition}  

\begin{proof}
	Let $D$ be a central division algebra of exponent $2$, and let $K/F$ be a field extension of degree at least $2$. Moreover, assume that $K/F$ is totally positive and that $K \hookrightarrow D$. By the Double Centralizer Theorem, we have  
\[
[D : F] = [K : F] \cdot [C_D(K) : F],
\]  
which implies that  
\[
[C_D(K) : K] = \left(\frac{\operatorname{Ind}(D)}{[K : F]}\right)^2.
\]  

	Let $\Delta$ be the underlying division algebra of $C_D(K)$. Then, for some $m$, we have  
\[
[C_D(K) : K] = (\operatorname{Ind}(\Delta))^2 m^2.
\]  
Therefore, $\operatorname{Ind}(\Delta)$ divides  
$
\frac{\operatorname{Ind}(D)}{[K : F]},
$
which implies that  
\[
\operatorname{Ind}(\Delta) \leq \frac{\operatorname{Ind}(D)}{[K : F]} < \operatorname{Ind}(D).
\]  

	Since $\operatorname{Ind}(D \otimes_F K) = \operatorname{Ind}(C_D(K)) = \operatorname{Ind}(\Delta)$, we obtain  
\[
\operatorname{pind}(D \otimes_F K) < \operatorname{pind}(D) = \operatorname{Ind}(D),
\]  
which contradicts the assumption that Becher's conjecture holds.  

\end{proof}
\begin{proposition}\label{becher conjecture particular case}
Let $F$ be a formally real field, and let $K/F$ be a Galois totally positive field extension. Suppose that $D$ is a central division algebra of index $2^n$ and exponent $2$, containing a subfield $L \subset F_{\rm py}$ such that $[L : F] = 2^{n-2}$. Then, we have  
\[
\operatorname{pind}(D) = \operatorname{pind}(D \otimes_F K),
\]  
which shows that Becher's conjecture holds for such central division algebras.  

\end{proposition}
\begin{proof}
By the Double Centralizer Theorem, we have  
\[
[D : F] = [L : F] \cdot [C_D(L) : F],
\]  
which gives  
\[
2^{2n} = 2^{2(n-2)} \cdot [C_D(L) : L].
\]  
Hence, it follows that  
\[
[C_D(L) : L] = 2^4.
\]  
Since $L \subset F_{\rm py}$, we have $L_{\rm py} = F_{\rm py}$ and $(KL)_{\rm py} = K_{\rm py}$. By the proof of Lemma \ref{finite Galois tp}, the extension $KL/L$ is Galois and totally positive. By \cite[Theorem 1.5]{MPS}, it follows that  
\[
\operatorname{pind}(C_D(L)) = \operatorname{pind}(C_D(L) \otimes_L KL).
\]  
Moreover, since  
\[
\operatorname{pind}(D) = \operatorname{pind}(C_D(L))
\]  
and  
\[
\operatorname{pind}(D \otimes_F KL) = \operatorname{pind}(C_D(L) \otimes_L KL),
\]  
it follows that  
\[
\operatorname{pind}(D) = \operatorname{pind}(D \otimes_F KL).
\]  
Thus, we obtain the desired result.

\end{proof}

\bibliography{mybib}{}
\bibliographystyle{alpha}
\end{document}